\documentclass[11pt, reqno]{amsart}

\usepackage{amsmath,amssymb,epsfig}

\usepackage{fullpage}

\usepackage{amsfonts, amssymb, stmaryrd}
\usepackage[all]{xy}

\usepackage{latexsym}
\usepackage{cite}

\usepackage[usenames]{color}
\usepackage{url}
\usepackage{boxedminipage}
\usepackage{amsbsy}

\usepackage{hyperref}
\hypersetup{colorlinks}

\newcommand{\frs}{\mathfrak s}
\newcommand{\frt}{\mathfrak t}

\newcommand{\cQ}{{\mathcal Q}}
\newcommand{\cR}{{\mathcal R}}
\newcommand{\cS}{{\mathcal S}}
\newcommand{\cT}{{\mathcal T}}
\newcommand{\cZ}{{\mathcal Z}}
\newcommand{\N}{\mathbb{N}}
\newcommand{\R}{\mathbb{R}}

\newcommand{\Z}{\mathbb{Z}}

\usepackage{mathrsfs}

\newcommand{\Graph}{\mathsf{Graph}}

\newcommand{\RRel}{\mathsf{RR}}
\newcommand{\RR}{\mathsf{RR}}

\newcommand{\HS}{\mathsf{HyPh}}
\newcommand{\DS}{\mathsf{DS}}
\newcommand{\HDS}{\mathsf{HDS}}

\DeclareMathAlphabet{\mathpzc}{OT1}{pzc}{m}{it}

\newcommand{\toto}{\rightrightarrows}

\newtheorem{thm}{Theorem}[section]

\newtheorem{theorem}[thm]{Theorem}

\newtheorem{proposition}[thm]{Proposition}
\newtheorem*{corollary*}{Corollary}

\theoremstyle{definition}
\newtheorem{definition}[thm]{Definition}
\newtheorem{remark}[thm]{Remark}
\newtheorem{example}[thm]{Example}
\newtheorem{notation}[thm]{Notation}

\numberwithin{equation}{section}

\begin{document}
\title{A category  of hybrid systems.}  
\author{Eugene Lerman }

\begin{abstract}
  We propose a definition of the category of hybrid systems in which
  executions are special types of morphisms. 
  Consequently
  morphisms of hybrid systems send executions to executions.

We plan to use this result to define and study networks of hybrid systems.

\end{abstract}
\maketitle
\tableofcontents

\section{Introduction}

In this paper propose a definition of a category of 
non-deterministic hybrid systems.  Hybrid systems are dynamical systems
that exhibit both continuous time evolution, which we model by vector
fields on manifolds with corners, and abrupt transitions (``discrete transitions'' or ``jumps'').

Our basic philosophy is that of category theory --- so rather than
study dynamical systems one at a time we aim to study maps between all
relevant dynamical systems at once.  To quote
Silverman~\cite{Silverman.book}: 
\begin{quote}
``A meta-mathematical principle is
that one first studies (isomorphism classes of) objects, then one
studies the maps between objects that preserve the objects'
properties, then the maps themselves become objects for study...''
\end{quote}

Definitions of a hybrid dynamical systems varies widely in literature
depending on the background and needs of the practitioners.  They all
include directed graphs, phase spaces attached to the nodes of the
graph and partial maps or relations attached to arrows of the graph.

To get started we choose one definition of a directed graph.
We recall a fairly standard definition of a hybrid dynamical system
and its executions.  We introduce the notion of a hybrid phase space
so that we can think of a hybrid dynamical system as a hybrid phase
space with a ``hybrid vector field.''  We propose a notion of a map
between two hybrid systems. This turns hybrid systems into a
category. We justify our definition by proving that maps of hybrid
systems send executions to executions.  We also explain why executions
can be thought of as morphisms of hybrid dynamical systems.  We are
aware of one previous attempt to bring category theoretic methods to
hybrid dynamical systems by Ames \cite{AmesThesis} Ames and Sastry
\cite{AmesSastry}.  Our construction is simpler and covers a larger
class of systems.  Readers who like category theory may be entertained
by the appearance of pseudo-double categories.

We plan to use our approach to hybrid systems to define and study networks of hybrid systems. In particular our goal is to prove for hybrid systems analogues of results in \cite{DeVille.Lerman2},   \cite{DeVille.Lerman}, \cite{VSL}, \cite{LS}.\\

\noindent {\bf Acknowledgments:} I thank Sayan Mitra for many hours of
conversations.  In a better world we would have written this paper together.

\section{Background}

In this section we review the definitions of directed graphs,
manifolds with corners (which we call ``regions'') and smooth maps
between regions, set-theoretic relations and a traditions definition
of a hybrid dynamical system essentially following  \cite{SJLS}.
\subsection*{Graphs}
We start by fixing a notion of a graph and of a map of graphs.  

\begin{definition} \label{def:graph} A {\sf directed multigraph} $A$
  is a pair of collections $A_0$ (nodes, vertices) and $A_1$ (arrows,
  edges) together with two maps $\frs,\frt:A_1\to A_0$ (source and
  target).  We do not require that $A_1, A_0$ are sets in the sense of ZFC.

We depict an arrow $\gamma \in A_1$ with the source $a$ and target $b$
as $a\xrightarrow{\gamma}b$.  We write $A=\{A_1\toto A_0\}$ to
remind ourselves that our graph $A$ consists of two collections  and two maps.

A graph $A=\{A_1\toto A_0\}$ is {\sf finite} if the collections $A_1$ and
$A_0$ are finite.
\end{definition}

\begin{remark}
  Every category has an underlying graph: forget the composition of
  morphisms.  Since the collections of objects and morphisms in a
  given category may be too big to be sets, we choose to define
  graphs in such a way as to induce the underlying graphs of
  categories that are not necessarily small.  This causes no problems.
\end{remark}
Next we record our definition of a map of graphs:
\begin{definition}\label{def:map_of_graphs}
  A {\sf map of graphs} $\varphi:A\to B$ from a graph $A$ to a graph
  $B$ is a pair of maps $\varphi_1:A_1\to B_1$, $\varphi_0:A_0\to B_0$
  taking edges of $A$ to edges of $B$, nodes of $A$ to nodes of $B$ so
  that for any edge $\gamma$ of $A$ we have
\[
\varphi_0 (\frs(\gamma)) = \frs(\varphi_1 (\gamma))\quad
\textrm{and} \quad \varphi_0 (\frt(\gamma)) = \frt(\varphi_1 (\gamma)).
\]
We will usually omit the indices 0 and 1 and write $\varphi(\gamma)$
for $\varphi_1 (\gamma)$ and $\varphi(a)$ for $\varphi_0 (a)$.
\end{definition}

Note the maps of graphs can be composed and that the composition is associative.
Hence graphs form a category.  We now formally record its definition.

\begin{definition}
  Directed multigraphs form a category $\Graph$. Its objects are directed
  graphs (see Definition~\ref{def:graph}).  Morphisms are maps of
  graphs (see Definition~\ref{def:map_of_graphs}).
\end{definition}

\subsection*{Regions and continuous time dynamics}
Continuous time dynamics takes place in {\em regions}.  For most
purposes of this paper one may take a region to be an open subset of
some coordinate space $\R^n$.  However, we also like to consider
closed intervals $[a,b]\subset \R$ as regions.  On the other hand, we
do not want to consider {\em arbitrary} subsets of $\R^n$ as regions
--- those are too wild and we will not be able to make sense of
differential equations on such sets.  A subset $D$ of $\R^n$ (for
some $n$) with smooth boundary should definitely be considered a
region.  For many purposes it is not very wrong to think of regions
this way.  However, later on we will need to take products of regions.
The product of two regions with smooth boundary no longer has a smooth
boundary.  For example the product of two unit intervals is the unit square:
\[
[0,1]\times[0,1]= \{ (x,y) \in \R^2 \mid 0\leq x \leq 1, 0\leq y \leq 1\}.
\]
The boundary of the unit square is only piecewise smooth.  This forces
us to define a region to be a subset of $\R^n$ with smooth corners.
(Recall that a corner of an $n$-dimensional region is smooth if it is
locally diffeomorphic to the standard orthant $[0,\infty)^n$.)  It is
convenient for various purposes to treat regions as abstract manifolds
with corners.  However, a reader will not be too far wrong simply to
think of a region as an open subset of some $\R^n$'s with a
piecewise-smooth boundary. There are a number of textbooks and survey
articles that deal with manifolds with corners.  We recommend Joyce
\cite{Joyce} and Michor \cite{Michor}. Our notion of a map of
manifolds with corners follows \cite{Michor} and is much weaker than
the one in \cite{Joyce}.  Namely we only require that a smooth map
between manifolds with corners pulls back smooth functions to a smooth
functions.\footnote{That is, for the purpose of defining smooth maps
  we think of manifolds with corners as diffeological spaces \cite{Igl}.}
In particular we allow corners to be mapped into the interior.


\begin{definition}

  A {\sf vector field $X$ on a manifold with corners $D$} is a section
  of its tangent bundle $TD\to D$.  We write $X\in \Gamma (TD)$.  A
  {\sf integral curve} of the vector field $X$ is a smooth map $x:I\to
  D$ of $X$ (where $I$ is an interval) so that $\frac{d}{dt}x =
  X(x(t))$ for all $t\in I$.
\end{definition}

\begin{definition}
  A {\sf continuous time dynamical system} is a pair $(D,X)$ where $D$
  is a manifold with corners (i.e., a region) and $X$ is a vector field on $D$.
\end{definition}

Continuous time dynamical systems form a category. 

Namely we define a map from a (continuous time) dynamical
system $(D_1, X_1)$ to a dynamical system $(D_2, X_2)$ to be a map
$f:D_1\to D_2$ of manifolds with corners with $Tf \circ X_1 =X_2 \circ
f$ (here and elsewhere in the paper $Tf$ denotes the differential of $f$).

\begin{remark}
  It is easy to see that if $f:(D_1, X_1) \to (D_2, X_2)$ and $g:(D_2,
  X_2) \to (X_3, D_3)$ are two maps of continuous time dynamical
  systems then so is their composite $g\circ f$.  Hence continuous
  time dynamical systems do form a category.  We denote it by $\DS$.
\end{remark}
\begin{definition}[the category of  $\DS$ continuous time dynamical
  systems]
  The objects of the category $\DS$ of continuous time dynamical
  systems are pairs $(D,X)$ where $D$ is a manifold with corners and
  $X$ is a vector field on $D$.  A morphism from $(D,X)$ to $(D',X')$
  is a map  $f:D\to D'$  of manifolds with corners with
\[
Tf \circ X =X' \circ f.
\]
\end{definition}

\begin{remark} 
  An integral curve  $x:[a,b]\to D$ of a dynamical system $(D, X)$ can be
  thought of a map of dynamical systems as follows. 
Recall that every interval $[a,b]\subset \R$ is equipped with the
constant vector field $\frac{d}{dt}$.   By definition
  the image of the vector field $\frac{d}{dt}$ by the map $x$
  is the derivative $\frac{dx}{dt}$:
\[
\left.\frac{dx}{dt}\right|_t: = Tx \left(\left.\frac{d}{dt}\right|_t\right).
\]
Since $x$ is an integral curve of $X$, $\frac{dx}{dt}|_t = X(x(t))$. Hence,
\[
Tx \circ \frac{d}{dt} = X \circ x.
\]
Thus a map of manifolds with corners $x:[a,b]\to D$ is an integral curve
of a vector field $X$ on $D$ if and only if $x:([a,b],
\frac{d}{dt})\to (D, X)$ is a morphism in the category $\DS$.
\end{remark}
\begin{remark}\label{rmrk:2.10}
  It $x:[a,b] \to D$ is a trajectory of a vector field $X$ then for
  any $b'<b$ the restriction $x|_{[a,b']}$ is also a trajectory of
  $X$.  For this reason we regard maps of
  the form $x:[a,a]\to D$ as integral curves of $X$.  Of course the
  closed interval $[a,a]$ is a single point, so the derivative of $x$
  in this case doesn't quite make sense.  None the less we will find
  this point of view convenient when we deal with executions of hybrid
  systems.
\end{remark}

\subsection*{Relations}

\begin{definition}[Relation] \label{def:rel}
Given two sets $X$ and $Y$ we call a subset $R$ of the product
$Y\times X$ a {\sf relation} and think of it as a ``generalized  map''
{\sf from $X$ to $Y$} (note the order!).  
\end{definition}

\begin{remark}
  The reason from why we think of $R\subset Y\times X$ as a map from
  $X$ to $Y$ and not from $Y$ to $X$ has to do with compositions of
  relations and of functions.  Namely, given two relations $S\subset
  Z\times Y$ and $R\subset Y\times X$ their composition $S\circ R$ is
  defined by
\[
S\circ R :=\{(z,x) \in Z\times X\mid \textrm{ there is } (z,y) \in S,
(y',x) \in R\textrm{ with } y=y'\}.
\]
If $f:X\to Y$ is a function, its {\sf graph} is the relation
\[
\mathrm{graph}(f):= \{(y,x) \in Y\times X\mid y= f(x)\}.
\]
With these definitions, given a function $g:Y\to Z$ we have
\[
\mathrm{graph}(g\circ f) = \mathrm{graph}(g)\circ \mathrm{graph}(f).
\]
\end{remark}
\begin{remark}
Note that if $R\subset Y\times X$ is a relation so that the
intersection $R \cap (Y\times \{x\})$ consists of {\em exactly} one
point for each $x\in X$ then $R$ is a graph of a function from $X$ to
$Y$.  If the intersection $R \cap (Y\times \{x\})$ consists of {\em at
  most} one point for each $x\in X$ then $R$ is a graph of a {\em
  partial function} from $X$ to $Y$ whose domain of definition is the
set $\{x\in X \mid R \cap (Y\times \{x\}) \not = \emptyset\}$.  We
will refer to the image of a relation $R\subset Y\times X$ under the
projection $\pi_X:Y\times X \to X$ as the domain of the relation $R$.
In the hybrid dynamical systems literature domains of relations and/or
partial maps are sometimes referred to as {\sf guards} and relations as
{\sf resets}.
 \end{remark}

\subsection*{Hybrid dynamical systems}
 We next recall the traditional definition of a hybrid dynamical
 system. It is 
 a slight variant of
 \cite[Definition~2.1]{SJLS}).  Note that in \cite{SJLS} what we call
 {\em manifolds with corners/regions} are called {\em domains}. 
 Since in mathematics and computer science
 literature the word ``domain'' has other meanings we prefer to use
 the word ``region.''  Another name for what we call ``regions'' is
 {\em invariants}.  But in mathematics an ``invariant'' has too many
 other meanings (e.g., invariant submanifolds, invariant functions,
 invariant vectors etc.).

\begin{definition}[Hybrid dynamical system]    
A {\sf hybrid dynamical system} (HDS) consists of 
\begin{enumerate}
\item A directed graph $A=\{A_1\toto A_0\}$;
\item For each node $a\in A_0$ a dynamical system  $(R_a, X_a)$ where $X_a$ is a vector field on the manifolds with corners $R_a$;
\item For each arrow $a\xrightarrow{\gamma}b$ of $A$  a {\sf reset } relation
  $R_\gamma \subset R_b \times R_a$.
\end{enumerate}
Thus a hybrid dynamical system is a tuple $(A =\{A_1\toto A_0\},
\{(R_a, X_a)\}_{a\in A_0}, \{R_\gamma \}_{\gamma \in A_1}))$.
\end{definition}

\begin{example}\label{ex:1}
  Here is an example of a very simple hybrid dynamical system.  We
  take $A$ to be the graph with one node $*$ and one arrow
  $*\xrightarrow{\gamma}*$ (formally $A_1=\{\gamma\}$, $A_0=\{*\}$ and
  $\frs(\gamma) = * = \frt(\gamma)$).  We assign to $*$ the constant
  vector field $\frac{d}{dx}$ on the unit interval $[0,1]$.  We take
  $\R_\gamma$ to be the one element set
\[
R_\gamma:= \{(0,1)\}.
\]
By our convention (Definition~\ref{def:rel}) it is the graph of a map
that takes the endpoint $1$ of the closed interval $[0,1]$ to the
endpoint $0$ (and not 0 to 1). Thus
\[
H:= \left\{\left\{
\{\gamma\}\toto \{*\}\right\}, ([0,1], \frac{d}{dx}), \{(0,1)\}
\right\}.
\]
We'll describe the dynamics of the system once we define executions.
\end{example}
\mbox{}\\

\subsection*{Executions.}
Having defined hybrid dynamical systems we now define the
corresponding dynamics.  The notion of an {\sf execution} (that is of
an ``integral curve'' or of a ``hybrid trajectory'') of a hybrid
dynamical system is supposed to captures the following idea.  Given a
hybrid dynamical system
\[
(A, \{(R_a, X_a)\}_{a\in A_0}, \{R_\gamma \}_{\gamma \in A_1}))
\] 
a hybrid trajectory would start at a point in some region $R_{a(1)}$.
For an interval of time $[t_0, t_1]$ it would follow an integral
curve $\sigma_{a(1)}$ of the vector field $X_{a(1)}$ until it reaches
a point $\sigma_{a(1)}(t_1)$ inside the domain of a relation
$R_{\gamma(1)}: R_{a(1)}\to R_{a(2)}$. Now the trajectory
is allowed to jump to a point $y$ in some region $R_{a(2)}$  with $(y,\sigma_{a(1)}(t_1)\in R_{\gamma(1)}$ and follow the
integral curve $\sigma_{a(2)}$ through the point $y$ of the vector field $X_{a(2)}$  for an
interval of time $[t_1, t_2]$. 
And so on for an
increasing sequence of times $\{t_0, t_1, t_2, \ldots..\}$, which may
be finite or infinite. 
This leads to the following definition, which is fairly standard.  We
will revisit the definition: see Definition~\ref{def:gen-execution}
below.

\begin{definition}[An execution with jump times indexed 
by the natural numbers $\N$] \label{def:execution}\mbox{}\\
  Let $H= (A =\{A_1\toto A_0\}, \{(R_a, X_a)\}_{a\in A_0}, \{R_\gamma
  \}_{\gamma \in A_1}))$ be a hybrid dynamical system.  An {\sf
    execution} of $H$ is
\begin{enumerate}
\item an nondecreasing sequence $\{t_i\}_{i\geq 0}$ of real numbers
\item a function $\varphi_0:\N\to A_0$;
\item a function $\varphi_1:\N\to A_1$ compatible with $\varphi_0$: we
  require that $\frs(\varphi_1(i)) = \varphi_0 (i)$ and
  $\frt(\varphi_1 (i)) = \varphi_0 (i+1)$;
\item an integral curve $\sigma_i: [t_{i-1}, t_{i}] \to R_{\varphi_0
    (i)}$ of $X_{\varphi_0(i)}$ 
(with $t_{-1}$ being some number less than $t_0$);
\item the terminal end point of $\sigma_i$ and the initial
  end point of $\sigma_{i+1}$ are related by the reset relation
  $R_{\varphi_1(i)}$: 
\[
(\sigma_{i+1} (t_{i}), \sigma_i (t_{i})) \in R_{\varphi_1(i)}.
\]
\end{enumerate}
\end{definition}
\begin{example}\label{ex:2}
  Consider the hybrid system $H$ of Example~\ref{ex:1}.  What would an
  execution of such a system look like?  We have no choice in defining
  the functions $\varphi_0$ and $\varphi_1$ since $A_0$ and $A_1$ are
  one point sets: we set $\varphi_0(n) = *$ and $\varphi_1(n)= \gamma$
  for all $n\in Z$. If we take $t_i =i$, then $\sigma_i:[i-1,i]\to [0,1]$
  is given by $\sigma_i(t) = t-i+1$. Therefore
\[
 (\{t_i\}_{i\in \N}, \varphi_0, \varphi_1, \{\sigma_i\}_{i\in \N})
\]
is an execution of $H$.
\end{example}

\begin{remark}
  If $t_{i} = t_{i-1}$ the condition in Definition~\ref{def:execution} 
that $\sigma_i: [t_{i-1}, t_{i}] \to
  R_{\varphi_0 (i)}$ is an integral curve of $X_{\varphi_0(i)}$ should be
  interpreted in the sense of Remark~\ref{rmrk:2.10}.  This amounts to
  saying that $\sigma_i (t_i) =\sigma_i (t_{i-1}) $ is a point of
  $R_{\varphi_0 (i)}$.  Note that the next conditions forces
  $(\sigma_{i+1}(t_{i}), \sigma_i (t_i) )\in R_{\varphi_1(i)}$. In other
words if $t_{i-1} = t_i$ the execution jumps.
\end{remark}

\begin{remark}\label{rmrk:2.20}
  More generally jump times of an execution may be indexed by a subset
  $S$ of the integers $\Z$ of the form $S=[n,m]$, $n\leq m$, $n,m\in
  \Z$, or by $S=(-\infty, N]$ by by $S= [N,+\infty)$ for some $N\in
  \Z$. Note that $S= \emptyset$ also makes sense: this is an execution
  that is simply an integral curve of a vector field.  
   We will give a different definition of an execution that
  includes all of these cases, see Definition~\ref{def:gen-execution} below.
\end{remark}

\section{Hybrid phase spaces}
If we forget the vector field of a continuous time dynamical system
$(D,X)$ we get a manifold with corners $D$, which we think of as the
{\em phase space } of our dynamical system.  Therefore it makes sense
to define a {\em hybrid phase space} to be a ``hybrid dynamical system
without the vector fields.''  Formally we record the following
definition, which we think is new.

\begin{definition}[Hybrid phase space] \label{def:hps} A {\sf hybrid
    phase space} consists of
\begin{enumerate}
\item A directed graph $A=\{A_1\toto A_0\}$;
\item For each node $a\in A_0$ a manifold with corners  $R_a$;
\item For each arrow $a\xrightarrow{\gamma}a'$ of $A$ a {\sf reset }
  relation $R_\gamma \subset R_{a'} \times R_a$.
\end{enumerate}
Thus a hybrid phase space is a tuple $(A =\{A_1\toto A_0\},
\{R_a\}_{a\in A_0}, \{R_\gamma \}_{\gamma \in A_1}))$.
\end{definition}

\begin{example}\label{ex:3}
  The underlying hybrid phase space of the hybrid dynamical system of
  Example~\ref{ex:1} consists of  the following data:
\begin{enumerate}
\item the  directed graph $A=\{\{\gamma\}\toto \{*\}\}$;
\item the region $R_* = [0,1]$;
\item the  { reset }
  relation $R_\gamma =\{(0,1)\}  \subset [0,1] \times [0,1]$.
\end{enumerate}
\end{example}

\begin{remark}
  The two collections $\{R_a\}_{a\in A_0}, \{R_\gamma \}_{\gamma \in
    A_1}$ in the definition of a hybrid phase space above look like
  the components of a map of directed graphs and they are.  To make this precise we need a definition.
\end{remark}

\begin{definition}  
  We define the graph $\RRel$ of {\sf regions and relations} as
  follows: the collection of nodes of $\RRel$ is the collection of all
  regions (i.e., manifolds with corners); the collection of all arrows
  of $\RRel$ is the collection of all (set-theoretic) relations
  between the regions.
\end{definition}
\begin{remark}
  The graph $\RRel$ is the underlying graph of a category whose
  objects are manifolds with corners and morphisms are relations
  between the underlying sets of manifolds with corners.
\end{remark}
With this definition and notational convention we can restate the
definition of a hybrid phase space as follows:
\begin{definition}[Hybrid phase space, version 2] \label{def:hps2} 
A {\sf hybrid phase space} is a map of directed graphs
\[
\cR: A\to \RRel.
\]
\end{definition}
\begin{example}\label{ex:3'}
  The underlying hybrid phase space of the hybrid dynamical system of
  Example~\ref{ex:1} is a map of graphs 
\[
\cR:\{\{\gamma\}\toto \{*\}\}\to \RRel 
\]
with $\cR(*) = [0,1]$ and $\cR(\gamma) = \{(0,1)\}$.
\end{example}

The following example will be important when we discuss executions as
maps of hybrid dynamical systems and when we re-define our notion of
an execution.
\begin{example}\label{remark:2.19.1}[Hybrid phase space associated
  with a nondecreasing sequence  $\{t_i\}_{i\in\N}$.]
 Define $\cZ$ to be the graph with the set of edges $\cZ_1:= \N$, the
  set of nodes $\cZ_0 := \N$ and the source and target maps given by
  $\frs (i) = i$, $\frt(i) = i+1$:
\[
0\xrightarrow{0}1 \xrightarrow{1}
\cdots \to i-1\xrightarrow{i-1}i\xrightarrow{i}i+1\xrightarrow{i+1}
i+2 \to\cdots .
\]
Let $\cT: {\cZ} \to \RRel$ be the map of graphs defined on vertices by
\[
\cT (i) =  [t_{i-1}, t_i]
\]
and on arrows by 
\[
\cT (i\xrightarrow{i}i+1) = \cT_i
\]
where $\cT_i :[t_{i-1}, t_i]\to [t_i, t_{i+1}]$ is the relation
consisting of one point $\{(t_i, t_i)\}\subset [t_i, t_{i+1}]\times
[t_{i-1}, t_i]$.  

\end{example}

Our definition of hybrid phase spaces as maps of graphs from arbitrary
graphs to $\RRel$ suggests the category of hybrid phase spaces could
be the slice category $\Graph/\RRel$, but this is too strict.  Note
that $\RRel$ has more
structure: 
in addition to the set-theoretic relations as morphisms we also have
smooth maps as morphisms between regions.  This suggests that we
should think of $\RRel$ as a double category \cite{BMM, Shul,
  Shulman10}.  Recall that double categories have two types of
1-arrows (``vertical'' and ``horizontal'') and, in addition, 2-cells
that are shaped like rectangles.  Composition is defined by pasting
the rectangles--- vertically and horizontally.

\begin{definition}[The double category $\RRel$ of manifolds with
  corners, smooth maps and set-theoretic relations] The double
  category $\RRel$ is defined as follows. Its objects are manifolds
  with corners.  The vertical 1-arrows are smooth maps. The horizontal
  1-arrow are set-theoretic relations. The 2-cells are diagrams of the
  form
\[
\xy
(-19,10)*+{X}="1";
(-19,-10)*+{Y}="2";
(15, 10)*+{X'}="3";
(15, -10)*+{Y'}="4";
{\ar@{->}_{f} "1";"2"};
{\ar@{->}^{R} "1";"3"};
{\ar@{->}_{S} "2";"4"};
{\ar@{->}^{f'} "3";"4"};
{\ar@{=>}^{} (-2,3); (-2,-3)};
\endxy 
\]
where $f,f'$ are smooth maps, $R \subset X'\times X$, $S\subset
Y'\times Y$ are  relations satisfying 
\[
(f',f)(R)\subset S.
\]
\end{definition}

\begin{definition}[A category of hybrid phase spaces $\HS$] \label{def:hps}
  The objects of the category $\HS$ are maps of graphs $\cR :A\to
  \RR$ with target $\RR$.  A morphisms from $\cR:A\to RR$ to $\cQ:B\to \RR$ 
 ``2-commuting'' triangle of the form
\begin{equation}
\xy
(-10, 10)*+{A} ="1"; 
(10, 10)*+{B} ="2";
(0,-2)*+{\RRel }="3";
{\ar@{->}_{\cR} "1";"3"};
{\ar@{->}^{\varphi} "1";"2"};
{\ar@{->}^{\cQ} "2";"3"};
{\ar@{<=}_{\scriptstyle \alpha} (4,6)*{};(-0.4,4)*{}} ; 
\endxy .
\end{equation}
That is, $\varphi:A\to B$ is a map of graphs and $\alpha$ assigns to each node  $a$ of the graph $A$ a  map of manifolds with corners
\[
\alpha_a: \cR(a) \to \cQ(\varphi(a))
\] 
so that for each arrow $a_1\xrightarrow{\gamma} a_2$ in the graph $A$,
we have a 2-cell
\[
\xy
(-20, 10)*+{\cR(a_1)}="1"; 
(-20,-10)*+{\cQ(\varphi(a_1))}="2";
(20, 10)*+{\cR(a_2)}="3"; 
(20,-10)*+{\cQ(\varphi(a_2))}="4";
{\ar@{->}_{ \alpha_{a_1}}"1";"2"};
{\ar@{->}^{\cR(\gamma)} "1";"3"};
{\ar@{->}_{\alpha_{a_2}} "3";"4"};
{\ar@{->}_{\cQ(\varphi(\gamma))} "2";"4"};
{\ar@{=>} (0,2)*{};(0,-2)*{}} ;
\endxy
\]
in $\RRel$.  Note that the latter condition amounts to the inclusion
\begin{equation}\label{eq:3.1}
(\alpha_{a_1}, \alpha_{a_2})(\cR(\gamma))\hookrightarrow \cQ(\varphi(\gamma)).
\end{equation}
The composition of morphisms is given by pasting of triangles:
\[
\xy
(10, 6)*+{B} ="1"; 
(-10, 6)*+{C} ="2";
(0,-6)*+{\RRel }="3";
{\ar@{->}^{\cQ} "1";"3"};
{\ar@{->}_{\psi} "1";"2"};
{\ar@{->}_{\cS} "2";"3"};
{\ar@{=>}_{\scriptstyle \,\beta} (1, 0.8)*{};(-4, 3)*{}} ; 
\endxy 
\quad \circ \quad 
\xy
(10, 6)*+{A} ="1"; 
(-10, 6)*+{B} ="2";
(0,-6)*+{\RRel }="3";
{\ar@{->}^{\cR} "1";"3"};
{\ar@{->}_{\varphi} "1";"2"};
{\ar@{->}_{\cQ} "2";"3"};
{\ar@{=>}_{\scriptstyle  \,\alpha} (1, 0.8)*{};(-4, 3)*{}} ;
\endxy 
\quad = \quad 
\xy
(10, 6)*+{A} ="1"; 
(-10, 6)*+{C} ="2";
(0,-6)*+{\RRel }="3";
{\ar@{->}^{\cR} "1";"3"};
{\ar@{->}_{\psi \circ \varphi} "1";"2"};
{\ar@{->}_{\cQ} "2";"3"};
{\ar@{=>}_{\scriptstyle  \,\delta} (1, 0.8)*{};(-4, 3)*{}} ;
\endxy 
\]
 where
\begin{equation}\label{eq:3.2}
\delta (a) := \beta_{ \varphi(a)} \circ \alpha_a :
\cR(a) \to \cS (\psi(\varphi(a))) 
\end{equation}
for all nodes $a$ of $A$.
\end{definition}

\begin{remark}
  Equation~\eqref{eq:3.2} strongly suggests that we should view
  $\RRel$ as having more structure than just a double category.
  Namely the horizontal category of $\RRel$ should really be view as
  the strict 2-category of regions, set-theoretic relations and
  inclusions of relations.
\end{remark}

\begin{remark}
  We will see in later work that the monoidal structure on the double
  category $\RRel$, which we ignore in this paper, is important in
  building a theory of networks of hybrid systems.
\end{remark}
\begin{example}
  The execution of Example~\ref{ex:2} is a map of hybrid phase spaces.
  This can be seen as follows.  The source hybrid phase space is the
  map $\cT:\cZ\to \RRel$ with $\cT(i) =[i-1,i]$ for all $i\in \N$
  (q.v.\ Example~\ref{remark:2.19.1}).  The target hybrid phase space
  is the phase space $\cR:\{\{\gamma\}\toto \{*\}\}\to \RRel$ of
  Examples~\ref{ex:1}, \ref{ex:3'}.  The desired map of hybrid phase
  spaces consists of the map of graphs 
\[
\varphi: \cZ\to \left\{\{\gamma\}\toto \{*\}\right\}
, \quad \varphi(i-1\xrightarrow{i-1}
i) = *\xrightarrow{\gamma}* \quad \textrm{ for all } i,
\]
and of the collection of smooth maps of closed intervals
\[
\{\sigma_i:\cT(i) = [i-1,i]\to [0,1]\mid \sigma_i (t) = t-i+1\}.
\]
\end{example}

\section{Hybrid dynamical systems}

We are now in position to redefine a hybrid dynamical system as follows.

\begin{definition}[Hybrid dynamical system, version 2] \label{def:hds2} 
A {\sf hybrid dynamical system} is a hybrid phase space
\[
\cR: A\to \RRel.
\]
together with a family of vector field $\{X_a \in \Gamma
(T\cR(a))\}_{a\in A_0}$, one for each region $\cR(a)$.  Thus a hybrid
dynamical system is a pair $(\cR: A\to \RRel,X= \{X_a\in \Gamma (T
\cR(a))\}_{a\in A_0})$.
\end{definition}

\begin{remark}\label{remark:2.19.2}
  Let $\{t_i\}_{i\in \N}$ be a nondecreasing sequence and $\cT: {\cZ}
  \to \RRel$ the corresponding hybrid phase space as in
  Example~\ref{remark:2.19.1}.  On each interval $\cT(i) = [t_{i-1},
  t_i]$ choose the constant vector field $\frac{d}{dt}|_ {[t_{i-1},
    t_i]}$.  Then $(\cT: {\cZ} \to \RRel, \{\frac{d}{dt}|_ {[t_{i-1},
    t_i]}\}_{i\in \N})$ is a
  hybrid dynamical system.
\end{remark}

\begin{definition}[Maps of hybrid dynamical systems]\label{def:mapHDS}
  A {\sf map from a hybrid dynamical system $(\cQ:{A}\to
    \RRel, X)$ to a hybrid dynamical system
    $(\cR:{B}\to \RRel, Y)$} is
\begin{enumerate}
\item a map of hybrid phase spaces $(\varphi, \{\alpha_a\}): \cQ\to \cR$ so that
\item 
\[
Y_{\varphi(a)}\circ \alpha_a  = T\alpha_a \circ X_a
\]
for all nodes $a\in A_0$.
\end{enumerate}
\end{definition}

\begin{remark} \label{remark:2.21}It is not hard to check that the
  composition of two maps of hybrid dynamical systems is a map of
  hybrid dynamical systems and that the composition is associative.
  Hence hybrid dynamical systems form a category, which we denote by
  $\HDS$.
\end{remark}
We are now in position to interpret executions as maps of hybrid
dynamical systems.  This reinterpretation allows us to broaden the
notion of an execution and to give a short proof of the main
result of this part of the paper: maps of hybrid dynamical systems
take executions to executions.

\begin{proposition}
  An execution of a hybrid dynamical system $(\cR: A\to \RRel,X)$ in
  the sense of Definition~\ref{def:execution} is a map $(\varphi,
  \{\sigma_i\}_{i\in\N}:(\cT: {\cZ} \to \RRel, )\to (\cR: A\to \RRel,
  X)$, where $(\cT: {\cZ} \to \RRel, \{\frac{d}{dt}|_ {[t_{i-1},
    t_i]}\}_{i\in \N}$ is a hybrid dynamical system defined in
  Remark~\ref{remark:2.19.2}.
\end{proposition}

\begin{proof}
Compare Definitions~\ref{def:execution} and \ref{def:mapHDS}.
\end{proof}
We now extend the notion of an
execution to allow for indexing of jump times by various subsets of
the integers (q.v.\ Remark~\ref{rmrk:2.20}).  We start by defining an
appropriate generalization of the hybrid dynamical system $(\cT: {\cZ}
\to \RRel, \{\frac{d}{dt}|_ {[t_{i-1},
    t_i]}\}_{i\in \N})$ of Remark~\ref{remark:2.19.2}.

\begin{definition}[hybrid time dynamical system] Let $Z$ be a directed
  tree with countably many vertices and no branching. That is, $Z$ is
  a directed graph such that for {\em any} two distinct nodes $x, y$
  of $Z$ there exists a unique directed path in $Z$ either from $x$ to
  $y$ or from $y$ to $x$ (but not both).

  Let $\cT:{Z} \to \RRel$ be a map of graphs with the
  following properties:
\begin{enumerate}
\item For any node $i$ of $Z$
\[
\cT(i) = [t_i^-, t_i^+]
\]
for some $t_i^-, t_i^+\in \R$ with $t_i^-\leq t_i^+$.

\item For any edge $i\xrightarrow{\gamma}j$ of $Z$ we have $t_i^+=
  t_j ^-$ and
\[
\cT(\gamma) =\{(t_j^-, t_i^+)\}:  [t_i^-, t_i^+]\to  [t_j^-, t_j^+],
\]
is a relation from $\cT(i)$ to $\cT(j)$.  
\end{enumerate}
On each interval $\cT(i)$ choose the constant vector field $X_i =
\frac{d}{dt}$.  We define a  hybrid dynamical system of the form 
\[
(\cT: {\cZ} \to
\RRel, \partial_t:=\{\frac{d}{dt}|_{ \cT(i)  }\}_{i\in \N})
\]
to be a {\sf hybrid time dynamical system}.  We think of such a system
as being analogous to the system $((a,b), \frac{d}{dt})$ in continuous
time dynamics.
\end{definition}

\begin{definition}[An execution of a hybrid dynamical system] 
\label{def:gen-execution}
We define an  {\sf  execution} of a hybrid system
$(\cQ:A \to \RRel, X)$  
to be a map of hybrid dynamical systems
\[
(\varphi, \{\sigma_i\}_{i\in Z_0}):
(\cT:Z\to \RRel, \partial_t)\to  
(\cQ:A\to \RRel, X).
\]
where $(\cT:Z\to \RRel, \partial_t)$ is a hybrid time dynamical system.
\end{definition}
\begin{notation}
  We abbreviate a hybrid dynamical system $(\cQ:A \to \RRel, X)$ as
  $(\cQ, A, X)$.
\end{notation}

We now obtain the following useful theorem:
\begin{theorem}
  \label{thm:2.23} 
Let $(\psi, \{\alpha_a\}): (\cR, A, X)\to (\cQ, B, Y)$ be a map of
hybrid dynamical systems and $(\varphi, \{\sigma_i\}_{i\in Z_0}):
(\cT, \partial_t)\to (\cQ, X)$ be an execution of the first system.
Then the composite morphism
\[
(\psi, \{\alpha_a\})\circ (\varphi, {\sigma_i}): (\cT, \partial_t)
 \to
(\cQ, B. Y)
\]
is an execution of the second system.  In other words morphisms of
hybrid dynamical systems send executions to executions.
\end{theorem} 
\begin{proof}
  By Definition~\ref{def:gen-execution} the composite morphism $(\psi,
  \{\alpha_a\})\circ (\varphi, {\sigma_i})$ is an execution of the
  system $(\cQ, B, Y)$.
\end{proof}

\begin{example}\label{ex:42}
  Let $A$ be the graph with one node and one arrow: $A=\{ \{\alpha,
  \}\toto \{*\}\}$.  Let $\cQ:A\to \RRel$ be the map of graphs
  with
\begin{equation*}
\cQ (*)=[0,1]^2 \quad \textrm{ and }\quad
\cQ (\alpha)=  \{((0,0),(1,1))\} .
\end{equation*} 
We think of the relation  $\cQ(\alpha)$ as a partial map sending $(1,1)$ to $(0,0)$.
Let $Y:[0,1]^2 \to \R^2 $ be a vector field of the form 
\[
Y(x,y) = (f(x,y), f(y,x))
\]
for some smooth function 
\[
f:[0,1]^2 \to \R.
\]  
Let
$\cR:\{\{\gamma\}\toto \{*\}\}\to \RRel$ be the hybrid phase space of
Example~\ref{ex:3'} and  $X:[0,1]\to \R$ be the vector field of the form
\[
X(x) = f(x,x),
\]
for the {\sf same} function $f$.  The map 
\[
\psi: \{\{\gamma\}\toto \{*\}\}\to A, \quad \varphi(*\xrightarrow{\gamma}*)
= *\xrightarrow{\alpha}* \, \in A
\]
is a map of graphs.  Define $\tau_*: \cR(*) = [0,1]\to \cQ(*) $ by
\[
\tau_* (x) = (x,x).
\] 
It is easy to see that 
\begin{equation}\label{eq:diag}
(\psi, \{\tau_*\}): (\cR, X)  \to (\cQ, Y) 
\end{equation}
is a map of hybrid dynamical systems.  It is also not hard to check
that for any execution $(\varphi, \{\sigma_i\}_{i\in Z_0}):(\cT: {\cZ}
\to \RRel, \partial_t)\to (\cR, X)$ the composite $(\psi, \tau_*)\circ
(\varphi, \{\sigma_i\}_{i\in Z_0})$ is an execution of $(\cQ, Y)$.
\end{example}

\begin{remark}
Both hybrid dynamical systems in Example~\ref{ex:42} are built out of
one open hybrid system.  The larger system $(\cQ, Y)$ is built by
interconnecting two copies of an open system.  The smaller system
$(\cR, X)$ is built by interconnecting inputs and outputs of the same
open system.  The existence of the map $(\psi, \{\tau_*\})$ can be
seen as being induced by a map of finite sets.  It is the one map from
the two element set $\{1,2\}$ to the one element set $\{0\}$.
We plan to address open hybrid systems, their interconnections and, more generally, networks of hybrid systems in future work.
\end{remark}

\section{Discussion} In this brief paper we introduced a new version
of the notion of a hybrid dynamical system.  We made it more compact.
We also introduced the notion of a morphism (``map'') of hybrid
dynamical systems.  This allowed us to turn hybrid dynamical systems
into a category.  Our notion of morphism also allowed us to view
executions as a particular kind of morphisms. Consequently morphisms
of hybrid dynamical systems take executions to executions even if the
hybrid systems in question are not deterministic!

In the present paper we viewed hybrid systems as generalizations of
continuous time dynamical systems.  There is another approach that
views hybrid systems as generalizations of automata and of labeled
transition systems (see, for example \cite{KLSV}).  One important
aspect of labeled transition systems is that of parallel composition
that allows one to synchronize parallel transitions by way of label
sharing.  Since our category of hybrid dynamical systems is built on
directed graphs, we fail to properly account for parallel composition
of hybrid systems.  However there is a  fix to this problem:
replace directed graphs in Definition~\ref{def:hps} by labeled
transition systems.  We plan to address this elsewhere.

\end{document}